\documentclass{article}
\usepackage[ps,dvips,arrow,matrix,tips]{xy}
\usepackage{amsmath,amssymb,amsthm}

\newcommand{\leg}[2]{\genfrac{(}{)}{}{}{#1}{#2}}
\newtheorem*{thma}{Theorem A}
\newtheorem*{thmb}{Theorem B}
\newtheorem{thm}{Theorem}[section]
\newtheorem{lem}[thm]{Lemma}

\theoremstyle{remark}
\newtheorem*{rmk}{Remark}

\newcommand{\R}{\mathcal{R}}
\renewcommand{\P}{{\mathbf{P}}}
\newcommand{\Pp}{\mathcal{P}}
\newcommand{\A}{{\mathbf{A}}}

\newcommand{\Falg}{\mathbf{F}}
\newcommand{\F}{\mathbf{F}}
\newcommand{\Ff}{\mathcal{F}}

\newcommand{\C}{{\mathcal{C}}}
\newcommand{\T}{{\mathcal{T}}}
\newcommand{\Gal}{\mathrm{Gal}}

\title{On quantitative analogues of the Goldbach and twin prime conjectures over $\F_q[t]$}
\author{Andreas O. Bender \\
School of Mathematics\\
Korea Institute for Advanced Study \\Seoul 130-722\\Republic of Korea\\
\texttt{andreas@kias.re.kr}
\and 
Paul Pollack\thanks{The second author is supported by NSF award DMS-0802970.}\\
School of Mathematics\\Institute for Advanced Study \\Einstein Drive\\Princeton, NJ 08540, USA\\
\texttt{pppollac@illinois.edu}
}
\date{September 8th, 2009}

\begin{document}
\maketitle
\begin{abstract} We study the number of ways to decompose a monic $F\in \F_q[t]$ of degree $n$ as a sum of two monic irreducible polynomials in $\F_q[t]$. Our principal result is an asymptotic formula for the number of such representations in the case when $q$ is large compared to $n$. In its range of validity, this formula agrees with what is suggested by heuristic arguments from the rational setting. We also present similar results towards an analogue of the twin prime conjecture.

\end{abstract}
\section{Introduction} In this article we consider analogues of the Goldbach and twin prime conjectures in the setting of polynomials over finite fields. Both conjectures then become two-parameter problems, with one parameter the size $q$ of the finite field and the other the degree of the polynomial. We prove that every polynomial of a given degree $n$ has as many representations as a sum of two irreducible polynomials as heuristic arguments lead one to expect, provided that $q$ is odd and large in comparison to $n$. We also show, under similar restrictions on $q$ and $n$, that the number of twin prime pairs of a given degree $n$ is as expected. Our results leave open the seemingly very hard question of what happens over a fixed finite field $\F_q$.

Define a \emph{Goldbach representation} of $F$ to be a way of writing $F$ in the form $F_1 + F_2$, where $F_1$ and $F_2$ are monic irreducible polynomials in $\F_q[t]$ with $\deg{F_1} = \deg{F}-1$ and $\deg{F_2} = \deg{F}$.\footnote{Rather than prescribing that $\deg{F_1} = \deg{F}-1$, one could permit any $F_1$ with $\deg{F_1} < \deg{F}$, which is how the ternary Goldbach problem is handled in \cite{EH91book}. However, in the range of $q$ and $n$ of interest to us, that choice turns out to affect neither the conjectured asymptotic \eqref{eq:rfqapprox} for the number of Goldbach representations nor the statement of our Theorem \ref{thm:goldbach}.} We write $R(F;q)$ for the number of Goldbach representations of $F$.
It is easy to formulate a guess as to how large $R(F;q)$ should be; heuristic arguments which are familiar from the theory of rational primes (see, e.g., \cite[\S1.2.3]{CP05}) suggest that for a monic polynomial $F$ of degree $n$, we should have
\begin{equation}\label{eq:rfqapprox}
 R(F;q) \approx \frac{q^{n-1}}{n(n-1)} \prod_{P \mid F}\left(1- \frac{1}{|P|}\right)^{-1}\prod_{P \nmid F}\left(1-\frac{1}{(|P|-1)^2}\right) \end{equation}
for a wide range of parameters.\footnote{Here we write $|A|$ for $q^{\deg{A}}$, so that $|A| = \#\F_q[T]/(A)$ for each nonzero polynomial $A$.} In fact, it is plausible  to conjecture that the left and right-hand sides in the formula above are asymptotic to one another in any range of the $q$ and $n$ space where $q^n\to\infty$ and $n \geq 3$, uniformly in $F$. (We include the restriction $n\geq 3$ because in characteristic $2$, polynomials of the form $T^2+T+\alpha$ do not have a Goldbach representation.) A similar conjecture was proposed by the second author as Conjecture 7.1.1 in \cite{pollack08thesis}.

It seems very difficult to prove anything in the direction of \eqref{eq:rfqapprox} when working over a fixed finite field $\F_q$. However, the situation improves if we allow $q$ to vary. The following two theorems are due to the first author (see \cite{bender08, bender08A}, and cf. \cite[Chapter 7]{pollack08thesis}):

\begin{thma} Let $\F_q$ be a fixed finite field of odd cardinality $q$, and let $F$ be a monic polynomial in $\F_q[t]$ of degree $n\geq 2$. If $s$ is sufficiently large, then $F$ has a Goldbach representation in $\F_{q^s}[t]$.
\end{thma}

\begin{thmb} Let $\F_q$ be a finite field of odd cardinality $q$, and let $F$ be a monic polynomial in $\F_q[t]$ of degree $n\geq 2$. If $q > 8(n+6)^{2n^2}$, then $F$ has a Goldbach representation in $\F_q[t]$.
\end{thmb}

Our first result is that if $q$ is large compared to $n$, then $R(F;q)$ is precisely as large as predicted by \eqref{eq:rfqapprox}.

\begin{thm}\label{thm:goldbach} Let $\F_q$ be a finite field of odd order $q$. Suppose that $F \in \F_q[t]$ is a monic polynomial of degree $n \geq 2$. Then \[ R(F;q) = \frac{q^{n-1}}{n(n-1)} + O(n! (n-1)! q^{n-3/2}) + O(n \cdot 2^{\binom{n+2}{2}} q^{n-2}),\]
where the implied constants are absolute.
\end{thm}

After a short computation, one obtains from Theorem \ref{thm:goldbach} that for odd $q$,
\begin{equation}\label{eq:rfqasymptotic} R(F;q) \sim \frac{q^{n-1}}{n(n-1)} \quad\text{whenever} \quad \frac{q}{n^3 2^{\binom{n+2}{2}}} \to \infty.\end{equation}
It may not be immediately obvious that this asymptotic agrees with our prediction \eqref{eq:rfqapprox}, since the two products in \eqref{eq:rfqapprox} are absent in \eqref{eq:rfqasymptotic}, but in our range of $q$ and $n$, both of these products are easily verified to be tending towards $1$.

We can prove similar results towards a polynomial version of the twin prime conjecture. Let $A$ be a polynomial in $\F_q[t]$. Barring congruence obstructions, one expects that there are infinitely many monic polynomials $F \in \F_q[t]$ for which both $F$ and $F+A$ are irreducible. If $A$ is a constant polynomial, this can be proved by a method of C. Hall (see \cite[Theorem 1.2]{pollack08}), but it seems very difficult to confirm this prediction for other values of $A$.

If $A \in \F_q[t]$ and $\deg{A} < n$, we write $\pi_2(n;A,q)$ for the number of monic polynomials $F$ of degree $n$ for which both $F$ and $F+A$ are irreducible. Here heuristics suggest (cf. \cite[Conjecture 1]{pollack08T}) that
\begin{equation}\label{eq:twinheuristic} \pi_2(n;A,q) \approx \frac{q^n}{n^2} \prod_{P \mid A}\left(1-\frac{1}{|P|}\right)^{-1} \prod_{P \nmid A}\left(1 - \frac{1}{(|P|-1)^2}\right).\end{equation}
We prove the following two results:

\begin{thm}\label{thm:twin} Let $\F_q$ be a finite field of odd cardinality $q$. Let $n$ be a natural number, and let $A$ be a nonzero polynomial in $\F_q[t]$ of degree smaller than $n$. If $s$ is sufficiently large, then one can always find a monic polynomial $P\in \F_{q^s}[t]$ of degree $n$ for which both $P$ and $P+A$ are irreducible over $\F_{q^s}$.
\end{thm}

\begin{thm}\label{thm:twin2} Let $\F_q$ be a finite field of odd cardinality $q$. Let $n$ be a natural number, and let $A$ be a nonzero polynomial in $\F_q[t]$ of degree smaller than $n$. Then
\[ \pi_2(n;A,q)  = \frac{q^n}{n^2}+ O(n!^2 q^{n-1/2}) + O(n \cdot 2^{\binom{n+2}{2}} q^{n-1}), \]
where the implied constants are absolute.
\end{thm}

Of course Theorem \ref{thm:twin} is contained in Theorem \ref{thm:twin2}, but the strategy of proof suggests to prove the two results separately. Theorem \ref{thm:twin2} gives us an asymptotic formula for $\pi_2(n;A,q)$ in the range of $q$ and $n$ indicated in \eqref{eq:rfqasymptotic}, and this formula agrees with the prediction \eqref{eq:twinheuristic} in this range of $q$ vs. $n$.

A version of Theorem \ref{thm:twin2} appeared in the second author's Ph. D. thesis \cite[Theorem 7.1.4]{pollack08thesis} under a sharper restriction on the characteristic of $\F_q$.

\subsection*{Notation} We write $p$ for the characteristic of $\F_q$ and use $\Falg$ for a fixed algebraic closure.

\section{The polynomial Goldbach problem}
\subsection*{Proof of Theorem \ref{thm:goldbach}}\label{sec:goldbach}

For certain details we shall refer the reader to \cite{bender08A}. Let $\Ff$ be the set of polynomials $f(x,t) \in \F_q[x,t]$ of total degree $n-1$ for which $f(t+b,t)$ is a monic, degree $n-1$ polynomial in $t$ for every choice of $b \in \F_q$. Then (see \cite[p. 7]{bender08A}) we have $\#\Ff = q^{I}$, where $I:= \binom{n+1}{2}-1$.

For every monic polynomial $g \in \F_q[t]$ of degree $n-1$, let $N_g$ be the number of ordered pairs $(f(x,t),b)$ with $f(t+b,t)=g$, where $f \in \Ff$ and $b \in \F_q$.  Then we have the following simple observation.

\begin{lem}\label{lem:ngconst} For every monic polynomial $g\in \F_q[t]$ of degree $n-1$, we have
\[ N_g = q^{\binom{n+1}{2} - (n-1)} = q^{I-n+2}. \]
\end{lem}
\begin{proof} Let $b \in \F_q$. Suppose $c_1(x), \dots, c_{n-1}(x) \in \F_q[x]$ and that the degree of each $c_i$ does not exceed $n-1-i$. Then there is a unique choice of $c_0(x) \in \F_q[x]$ for which $f(x,t):= \sum_{i=0}^{n-1} c_i(x) t^{i}$  belongs to $\Ff$ and satisfies $f(t+b,t) = g(t)$, namely
\[ c_0(x) := g(x-b) - \sum_{i=1}^{n-1} c_i(x) (x-b)^{i}. \]
It follows that
\[ N_g = q  \cdot q^{1+2+\dots+(n-1)} = q^{n(n-1)/2+1} = q^{n(n+1)/2-(n-1)}, \]
as desired.
\end{proof}

Let $F$ be a univariate polynomial of degree $n\geq 2$ over $\F_q$, and let $\R$ be the set of degree $n-1$ Goldbach summands of $F$, i.e.,
\[ \R:=\{g \in \F_q[t]: \deg{g}=n-1, \text{ both $g$ and $F-g$ are monic irreducibles}\}. \]
Thus $R(F;q) = \#\R$. Lemma \ref{lem:ngconst} reduces Theorem \ref{thm:goldbach} to the following estimate for $\sum_{g\in \R}{N_g}$.
\begin{lem}\label{lem:sumrg} We have \[ \sum_{g \in \R} N_g = \frac{q^{I+1}}{n(n-1)} + O(N q^{I+1/2})
+ O(q^{I} 2^{\binom{n+2}{2}} n),\]
where $N:= (n-1)! n!$.
\end{lem}

The proof of Lemma~\ref{lem:sumrg} depends on the following lemma due to Fried \& Jarden (\cite[Proposition 6.4.8]{fj05}).
\begin{lem}[Explicit Chebotarev density theorem for degree $1$ primes]\label{lem:cheb} Suppose that $L/\F_q(u)$ be a finite, geometric Galois extension of degree $N$. Let $\C$ be a conjugacy class of $\Gal(L/\F_q(u))$ and let $\Pp$ be the set of unramified, degree $1$ primes of $\F_q(u)$ for which $\leg{L/\F_q(u)}{P}= \C$. Then, with $g$ denoting the genus of $L/\F_q$, we have
\[ \left|\#\Pp- \frac{\#\C}{N} q\right| \leq 2\frac{\#\C}{N} (gq^{1/2} + g + N).\]
\end{lem}
Note that the estimate differs from Proposition 6.4.8 in~\cite{fj05} since, in their notation, we have $k=1$ and therefore the second term in the first formula on page 120 is just the empty set.

We  also require the following lemma (\cite[Lemma 4.3.2]{pollack08thesis}; see also \cite[\S3]{BW05}).

\begin{lem}\label{lem:disc} Let $f(t,u) \in \F_q[t,u]$. Suppose that $f$ as a polynomial in $t$ is monic of degree $n$ and is irreducible and separable over $\F_q(u)$. Let $K$ be the splitting field of $f$ over $\F_q(u)$. Suppose that for the element $b \in \F_q$, the discriminant of the polynomial $f(t,b)$ is nonzero. Then $f(t,b)$ is irreducible over $\F_q$ precisely when the Frobenius conjugacy class $(K/\F_q(u),P_b)$ is the conjugacy class of an $n$-cycle in $\Gal(K/\F_q(u))$. Here $P_b$ denotes the prime of $\F_q(u)$ associated to the $(u-b)$-adic valuation.
\end{lem}

\begin{rmk} Suppose that $\deg_{u} f = m$. Then the $t$-discriminant of $f(t,u)$ is a polynomial in $u$ of degree not exceeding $(2n-1) m$, as one sees by applying the Sylvester determinant representation of the discriminant. Moreover, this $t$-discriminant is nonvanishing, since $f$ is separable by hypothesis. Therefore the hypothesis of the lemma excludes at most $(2n-1)m$ values of $b \in \F_q$.
\end{rmk}

\begin{proof}[Proof of Lemma \ref{lem:sumrg}] Recalling the definition of $N_g$, we find that
\begin{equation}\label{eq:ngsum} \sum_{g \in \R} N_g = \sum_{f \in \Ff} \sum_{\substack{b \in \F_q \\ f(t+b,t) \text{ and } F-f(t+b,t) \text{ both irreducible}}}{1}.
\end{equation} We shall show that when $q$ is large compared to $n$, the inner sum can be estimated fairly precisely for almost all choices of $f \in \Ff$.

Indeed, in \cite{bender08A} it is shown that for $q$ large compared to $n$, most $f \in \Ff$ are such that both  of the following statements hold:
\begin{itemize}
\item Let $K_1$ be the splitting field of $f(t+u,t)$ over $\F_q(u)$ and let $K_2$ be the splitting field of $F(t)-f(t+u,t)$ over $\F_q(u)$. Then $\Gal(K_1/\F_q(u))$ is the full symmetric group on $n-1$ letters and $\Gal(K_2/\F_q(u))$ is the full symmetric group on $n$ letters.
\item Let $L$ be the compositum of $K_1$ and $K_2$. Then $L/\F_q(u)$ is a geometric Galois extension, and the map
\begin{align}\label{eq:isomorphism}
\Gal(L/\F_q(u)) &\cong \Gal(K_1/\F_q(u)) \times \Gal(K_2/\F_q(u)), \\
\sigma &\mapsto (\sigma|_{K_1}, \sigma|_{K_2})\notag
\end{align}
is an isomorphism.
\end{itemize}
More precisely, this holds for all $f \in \Ff$ satisfying seven technical conditions (\cite[p. 4]{bender08A}, cf. the six conditions of \S\ref{sec:twins} below), which in total exclude
\begin{equation}\label{eq:exceptional}
 \ll n \cdot 2^{\binom{n+2}{2}} q^{I-1} \end{equation}
polynomials $f \in \Ff$ (see \cite[p. 8]{bender08A}). Let $\Ff'$ be the subset of $\Ff$ consisting of nonexceptional polynomials, so that \eqref{eq:exceptional} is an upper bound on $\#\Ff\setminus \Ff'$.

Suppose $f \in \Ff'$. By Lemma \ref{lem:disc} and the subsequent remark, we have that excluding $O(n^2)$ values of $b \in \F_q$, the polynomial $f(t+b,t)$ is irreducible precisely when $(K_1/\F_q(u), P_b)$ is the conjugacy class of an $(n-1)$-cycle. Similarly, excluding $O(n^2)$ values of $b$, we have that $F-f(t+b,t)$ is irreducible exactly when $(K_2/\F_q(u),P_b)$ is the conjugacy class of an $n$-cycle. This means that by \eqref{eq:isomorphism}, the simultaneous irreducibility of these polynomials is, excepting $O(n^2)$ values of $b$, equivalent to $(L/\F_q(u), P_b)$ belonging to a certain conjugacy class $\C$ of $\Gal(L/\F_q(u))$ of order $(n-2)! (n-1)!$.

We apply the version of the Chebotarev density theorem quoted above with this conjugacy class $\C$. Since $[L:\F_q(u)] = (n-1)! n! = N$, we find that for $f \in \Ff'$, the inner sum in \eqref{eq:ngsum} is
\[ \frac{q}{n(n-1)} + O\left(\frac{1}{n^2}(g q^{1/2} + g + N)\right) + O(n^2). \]
It can be shown that $g = 1 + N(n^2-2n)$ (see \cite[p. 9]{bender08A}), which implies a bound on the error terms of $O(N q^{1/2})$. Summing over $\Ff = \Ff' \cup (\Ff\setminus \Ff')$, we find that the sum in \eqref{eq:ngsum} is given by
\begin{multline*} \sum_{f \in \Ff'} \frac{q}{n(n-1)} + O(\#\Ff' N q^{1/2}) + O(q(\#\Ff\setminus \Ff')) \\
= \frac{q^{I+1}}{n(n-1)} + O\left(\frac{q}{n(n-1)} \#\Ff\setminus \Ff'\right) + O(N q^{I+1/2}) + O(q \#\Ff\setminus \Ff') \\ = \frac{q^{I+1}}{n(n-1)} + O(N q^{I+1/2}) + O(n \cdot 2^{\binom{n+2}{2}} q^{I}),\end{multline*}
using the trivial estimate $\#\Ff' \leq \#\Ff = q^{I}$ and the bound \eqref{eq:exceptional} for $\#\Ff\setminus\Ff'$.
\end{proof}

\section{The polynomial twin prime problem}\label{sec:twins}
\subsection*{Proof of Theorem \ref{thm:twin}}
The proof is very similar to that of Theorem A given in \cite{bender08A}. Let $\mathfrak X$ denote the family of all plane algebraic curves $f_{1(c)}(x,t)=0$ over $\Falg$ of degree at most $n$ in $\A^2$. Let $I = \binom{n+2}{2}-1$ and fix an embedding of $\mathfrak X$ by a map
$$\begin{array}{rcl}
{\mathfrak X} & \longrightarrow  & \A^{I+1}\times \A^{2},\\
f_{1(c)}      & \mapsto          & (c)\times \{f_{1(c)}(x,t)=0\},
\end{array}$$
where $(c)$ is the coefficient vector of $f_{1(c)}$. Let $\Ff$ be the family of affine curves of degree $n$ in $\A^2$ such that $f_{1(c)}(t+b,t)$ is monic for any $b$. The coefficients $(c)$ of such a curve correspond to a point on an affine subspace $H\subset \A^{I+1}$ of codimension one. In particular, we have that $\#\Ff = q^{I}$. For any element $f_{1(c)}\in{\Ff}$ with coordinates $(c)\in H_{\Falg}$, we set
$$
f_{2(c)}(x,t)  =  f_{1(c)}(x,t)+A(t).
$$
We consider the families of curves in $H_{\Falg}\times \P_{\Falg}^{2}$
$$
\xymatrix@C=8ex{
C_{1(c)}\ar[dr]^{\alpha_1} & & C_{2(c)} \ar[dl]_{\alpha_2} \\
  &H_{\Falg},&
}
$$
where $C_{i(c)}=\alpha^{-1}_i (c)$ denotes the Zariski closure of the affine curve $f_{i(c)}(x,t)=0$ in $\P^{2}_{\Falg}$. Let $\beta_i$ denote the rational map $C_i\rightarrow \P^1$ given by the projection from $M=(1,1,0)$ in homogenised coordinates $(x,t,z)$. The maps $\beta_1$ and $\beta_2$ are in fact morphisms since the point $M$ does not lie on $C_{1(c)}$ or $C_{2(c)}$. Indeed, we are assuming $(c)\in H$ and so monicity of the two polynomials $f_i(t+b,t)$ for any $b$ implies that $M\not\in C_i$.

Imitating the argument of \cite{bender08A}, we prove the following: For all sufficiently large $s$, there is a point $(c)$ in $H_{\F_{q^{s}}}(\F_{q^{s}})$ such that the following conditions are satisfied:
\begin{enumerate}
\item Both $C_{1(c)}$ and $C_{2(c)}$ are smooth.
\item Both $C_{1(c)}$ and $C_{2(c)}$ are absolutely irreducible.
\item The Gauss maps of both $C_{1(c)}$ and $C_{2(c)}$ are separable.
\item The morphisms $\beta_1$ and $\beta_2$ are generic (in the sense of \cite[Definition 2.2]{bender08A}).
\item No line $x=t+b$ is tangent to both $C_{1(c)}$ and $C_{2(c)}$.
\item The line at infinity is not tangent to $C_1$ or $C_2$.
\end{enumerate}
For each of the six conditions, the points $(c)$ satisfying it form an open subscheme of $H_{\Falg}$. We now check that each of these subschemes is nonempty.

For the smoothness condition (1), we consider two cases: First suppose that $p \nmid n$. Consider the polynomial $f_{1}:= 2x^n-t^n+d$, where $d \in \F_q$. A short calculation shows that $C_{1(c)}$ is smooth for any nonzero $d$, while $C_{2(c)}$ is singular for at most $n$ different values of $d$. So suppose that $p$ divides $n$, and write
\[ A(t) = a_m t^m + a_{m-1} t^{m-1} + \dots + a_1 t + a_0. \]
Consider
\[ f_1:= 2x^n + x - t^n + d t^{n-1}. \]
If $d \neq 0$ it is easily checked that $C_{1(c)}$ is smooth. The nonvanishing of $d$ also guarantees that $C_{2(c)}$ is smooth, unless $m=n-1$, in which case smoothness follows under the additional assumption that $d+a_m \neq 0$. Therefore in either case the points $(c) \in H$ whose fibres $C_{i(c)}$ are smooth curves form a nonempty open subscheme $A_1$ (say).

To show that the subscheme $A_2$ corresponding to condition (2) is nonempty, we have to exhibit an $f_{1(c)}$ with $(c) \in H$ for which both $f_{1(c)}$ and $f_{2(c)}:= f_{1(c)} + A(t)$ are absolutely irreducible. Consider
\[ f_1:= x^n + b_1 t + b_0, \]
where $b_0$ and $b_1$ are chosen so that $b_1 \neq 0$, $b_1 \neq -a_1$ and $b_0 = -a_0$. Then the Eisenstein criterion implies the absolute irreducibility of both $f_{1(c)}$ and $f_{1(c)} + A(t)$, the respective primes being $b_1t + b_0$ and $t$.

Now consider the condition (3) of separability of the Gauss maps. This corresponds to the subscheme $A_3$, for which we consider
\[ f_1:= 2x^n - x^2 t^{n-2}. \]
It follows that $A_3$ is nonempty since both $f_1{(c)}$ and $f_2{(c)}$ intersect the line at infinity in a point of multiplicity $2$.

Conditions (4)--(6) can be treated by the same argument used to treat conditions (4)--(7) in \cite[bottom of p.5]{bender08A}. The single change necessary is that we now require an example showing the nonemptiness of the open subscheme corresponding to the values $(c)$ for which the line at infinity is tangent to neither $C_{1(c)}$ nor $C_{2(c)}$. We may choose
\[ f_1:= 2x^{n-j}t^{j}-t^{n} \]
with $j=1$ if $p\mid n$ and $j=0$ otherwise.

To complete the proof of Theorem \ref{thm:twin} it is enough to apply the Chebotarev density theorem exactly as in the corresponding proof of Theorem A. We obtain that for $s$ sufficiently large, there is a $b_0 \in \F_{q^s}$ for which both $f_1(t+b_0, t)$ and $f_2(t+b_0,t) = f_1(t+b_0,t) + A(t)$ are irreducible. This means that $f_1(t+b_0,t)$ and $f_2(t+b_0,t)$ are our desired twin prime pair.

\subsection*{Proof of Theorem \ref{thm:twin2}}
With the proof of Theorem \ref{thm:twin} out of the way, we can prove Theorem \ref{thm:twin2} by an argument analogous to that given for Theorem \ref{thm:goldbach} above. Setting
\[ \T:= \{g \in \F_q[t]: \deg{g}=n, \text{ both $g$ and $g+A$ are monic irreducibles}\}, \]
we have $\pi_2(n;A,q) = \#\T$. As in \S\ref{sec:goldbach}, let $N_g$ be the number of pairs $(f(x,t), b)$ with $f \in \Ff$, $b \in \F_q$, and $f(t+b,t)=g$. Since now the polynomials $f \in \Ff$ have total degree $n$, we replace $n-1$ by $n$ in Lemma \ref{lem:ngconst} to obtain the following.

\begin{lem}\label{lem:ngconst2} For every monic polynomial $g \in \F_q[t]$ of degree $n$, we have $N_g = q^{\binom{n+2}{2}-n} = q^{I+1-n}$.
\end{lem}

To estimate the inner sum of
\begin{equation}\label{eq:twinsquant}
 \sum_{g \in \T} N_g = \sum_{f \in \Ff} \sum_{\substack{b \in \F_q\\ f(t+b,t)\text{ and } f(t+b,t) + A(t) \text{ both irreducible}}}1, \end{equation}
we observe that a calculation almost identical to the one appearing in \cite[\S3]{bender08A} shows that conditions (1)--(6) above are satisfied for all $f \in \Ff$ with $\ll n \cdot 2^{\binom{n+2}{2}} q^{I-1}$ exceptions. Let $\Ff'$ be the set of nonexceptional $f \in \Ff$. Then for each $f \in \Ff'$ both the following statements hold.
\begin{enumerate}
\item Let $K_1$ be the splitting field of $f(t+u,t)$ over $\F_q(u)$ and let $K_2$ be the splitting field of $f(t+u,t)+ A(t)$ over $\F_q(u)$. Then $\Gal(K_1/\F_q(u))$ and $\Gal(K_2/\F_q(u))$ are both the full symmetric group on $n$ letters.
\item Let $L$ be the compositum of $K_1$ and $K_2$. Then $L/\F_q(u)$ is a geometric Galois extension and the map \begin{align*}
\Gal(L/\F_q(u)) &\cong \Gal(K_1/\F_q(u)) \times \Gal(K_2/\F_q(u)), \\
\sigma &\mapsto (\sigma|_{K_1}, \sigma|_{K_2})\notag
\end{align*}
is an isomorphism.
\end{enumerate}
Proceeding as in the proof of Theorem \ref{thm:goldbach}, we obtain from Lemma \ref{lem:cheb} that for each $f \in \Ff'$ the inner sum in \eqref{eq:twinsquant} is
\begin{equation}\label{eq:innersumtwin}
 \frac{q}{n^2} + O\left(\frac{1}{n^2}(gq^{1/2} + g + N)\right) + O(n^2). \end{equation}
Here $N = [L:\F_q(u)] = n!^2$ and $g$ is the genus of $L$ over $\F_q$. By formula (4) of \cite{BW05} and the well-known formula for the genus of an irreducible smooth plane algebraic curve, we find after a short calculation that \[ g = 1 + N(n^2-n-1).\]
This shows that the $O$-terms in \eqref{eq:innersumtwin} are $O(Nq^{1/2})$. Now following the proof in \S\ref{sec:goldbach}, we find that the double sum on the right-hand side of \eqref{eq:twinsquant} is
\[ \frac{q^{I+1}}{n^2} + O\left(N q^{I+1/2}\right) + O\left(n \cdot 2^{\binom{n+2}{2}} q^{I}\right).\]
But from Lemma \ref{lem:ngconst2}, this double sum is precisely $q^{I+1-n} \#\T$. Theorem \ref{thm:twin2} follows.

\newpage
\bibliographystyle{amsplain}
\bibliography{twins}
\end{document}